\documentclass[11pt]{article}

\usepackage{amssymb,amsthm,amsmath}
\usepackage[utf8]{inputenc}

\newcommand{\R}{\mathbb{R}}
\newcommand{\Z}{\mathbb{Z}}
\newcommand{\T}{\mathbb{T}}

\newcommand{\p}[1]{\mathbb{P}\left( #1 \right)}

\newcommand{\call}[4]{\int_{#1}^{#2} {#3} \; \textrm{d} {#4}}

\newcommand{\mb}[1]{\mathbb{#1}}

\newcommand{\norma}[2]{\left\Vert #1 \right\Vert_{#2}}
\newcommand{\e}{\varepsilon}
\newcommand{\sufit}[1]{\left\lceil #1 \right\rceil }
\newcommand{\podloga}[1]{\left\lfloor #1 \right\rfloor }

\DeclareMathOperator{\dist}{dist}

\newtheorem{thm}{Theorem}
\newtheorem{lem}{Lemma}
\newtheorem{cor}{Corollary}

\theoremstyle{definition}
\newtheorem*{def*}{Definition}
\newtheorem*{question*}{Question}

\theoremstyle{remark}
\newtheorem*{rem*}{Remark}

\title{A note on certain convolution operators}

\author{Piotr Nayar \thanks{Research partially supported by NCN Grant no. 2011/01/N/ST1/01839.}, Tomasz Tkocz \thanks{Research partially supported by NCN Grant no. 2011/01/N/ST1/05960.} }
\date{}

\begin{document}

\maketitle

\begin{abstract}
In this note we consider a certain class of convolution operators acting on the $L_p$ spaces of the one dimensional torus. We prove that the identity minus such an operator is nicely invertible on the subspace of functions with mean zero.
\end{abstract}



\section{Introduction}\label{sec:intro}

Let $\T = \R/\Z$ be the one dimensional torus viewed as a compact group with the addition modulo 1, $x\oplus y = (x+y)\mod 1$, $x, y \in \R$ equipped with the Haar measure --- the unique invariant probability measure (the Lebesgue measure). To begin with, fix $1 \leq p \leq \infty$ and consider the averaging operator $U_t$ acting on $L_p(\T)$ (with the usual norm $\|f\| = \left(\int_\T |f|^p\right)^{1/p}$ for $p < \infty$, and $\|f\| = \textrm{ess\,sup}_\T |f|$ for $p = \infty$) 
\begin{equation}\label{opA}
	(U_t f)(x) = \frac{1}{2t} \call{-t}{+t}{f(x \oplus s)}{s}, \quad t\in (0,1).
\end{equation}
If $t$ is small, is the operator $I - U_t$ invertible, or, in other words, how much does $U_tf$ differ from $f$? Of course, averaging a constant function does not change it, but excluding such a trivial case, we get a quantitative answer.

\begin{thm}\label{thm1}
Let $t \in (0,1)$. There exists a universal constant $c$ such that for every $1 \leq p \leq \infty$ and every $f\in L_p(\mb{T})$ with $\int_{\mb{T}}f=0$ we have
\begin{equation}\label{eq:thm1}
	\norma{f-U_t f}{} \geq ct^2 \norma{f}{},
\end{equation}
where $\norma{\cdot}{}$ denotes the $L_p$ norm.
\end{thm}

Note that if $p$ was equal to $2$, then, with the aid of the Fourier analysis, the above estimate would be trivial. However, $\|\cdot\|$ is set to be the $L_p$ norm for some $1 \leq p \leq \infty$, the constant does not depend on $p$, therefore the situation is more subtle. 

When $p=1$, if we further estimate the left hand side of \eqref{eq:thm1} using the Sobolev inequality, see \cite{GT}, we obtain the following corollary.

\begin{cor}\label{cor}
Let us consider $t \in (0,1)$ and assume that $f$ belongs to the Sobolev space $W^{1,1}(\mb{T})$ with $\int_{\mb{T}}f=0$. Then we have
\begin{equation}\label{eq:sob}
	\call{\mb{T}}{}{ \left| f'(x)- \frac{f(x\oplus t)-f(x\oplus -t)}{2t} \right|  }{x} \geq ct^2 \call{\mb{T}}{}{|f(x)|}{x},
\end{equation}
where $c > 0$ is a universal constant.
\end{cor}

\begin{rem*}
Setting $t = 1/2$, inequality \eqref{eq:sob} becomes the usual Sobolev inequality, so \eqref{eq:sob} can be viewed as a certain generalization of the Sobolev inequality.
\end{rem*}

\begin{rem*}
Set $f(x) = \cos (2\pi x)$. Then $\norma{f-U_t f}{} = \norma{f}{}\left( 1 - \frac{1}{2\pi t}\sin (2\pi t) \right) \approx t^2\norma{f}{}$, for small $t$. Therefore, the inequality in Theorem \ref{thm1} is sharp in a sense.
\end{rem*}

\noindent In this note we give a proof of a generalization of Theorem \ref{thm1}. 
We say that a $\T$-valued random variable $Z$ is $c$-\emph{good} with some positive constant $c$ if $\p{Z \in A} \geq c|A|$ for all measurable $A \subset \T$. Equivalently, by Lebesgue's decomposition theorem it means that the absolutely continuous part of $Z$ (with respect to the Lebesgue measure) has a density bounded below by a positive constant.
We say that a real random variable $Y$ is $\ell$-\emph{decent} if $Y_1+\ldots +Y_\ell$ has a nontrivial absolutely continuous part, where $Y_1, Y_2, \ldots$ are i.i.d. copies of $Y$. Our main result reads 
\begin{thm}\label{thm2}
Given $t \in (0,1)$ and an $\ell$-decent real random variable $Y$, consider the operator $A_t$ given by
\begin{equation}\label{eq:defA}
 (A_t f)(x) = \mb{E}f(x \oplus tY).
\end{equation}
Then there exists a positive constant $c$ which depends only on the distribution of the random variable $Y$ such that for every $1 \leq p \leq \infty$ and every $f \in L_p(\mb{T})$ with $\int_\mb{T}f=0$ we have
\[
	\norma{f-A_t f}{} \geq ct^2 \norma{f}{},
\]
where $\norma{\cdot}{}$ denotes the $L_p$ norm.
\end{thm}

\begin{rem*}
One cannot hope to prove a statement similar to Theorem \ref{thm2} for purely atomic measures. Indeed, just consider the case $p = 1$ and let $Y$ be distributed according to the law $\mu_Y=\sum_{i=1}^\infty p_i \delta_{x_i}$. Then for every $\e>0$ and every $t \in (0,1)$ there exists $f\in L_1(\mb{T})$ such that $\norma{f-A_t(f)}{}<\e$ and $\norma{f}{}=1$. To see this take $N$ such that $\sum_{i=N+1}^\infty p_i < \e/4$ and let $f_n(x)=\frac{\pi}{2}\sin(2\pi nx)$. Then $\norma{f_n}{}=1$. Let $n_0 \geq 8\pi/\e$. Consider a sequence $\big( (\pi n t x_1 \mod 2\pi, \ldots, \pi n t x_N \mod 2\pi) \big)_n$ for $n=0,1,2,\ldots,n_0^N$ and observe that by the pigeonhole principle there exist $0 \leq n_1< n_2 \leq n_0^N$ such that for all $1 \leq i \leq N$ we have $\dist(\pi t x_{i}(n_1-n_2), 2\pi \mb{Z}) \leq \frac{2\pi}{n_0}$. Taking $n=n_2-n_1$ we obtain 
\begin{align*}
	\norma{f_n-A_t(f_n)}{} & \leq \frac{\pi}{2}\sum_{i=1}^N p_i \norma{\sin(2\pi nx)-\sin(2\pi n(x+tx_i))}{} + \frac{\e}{2} \\
	& = \pi \sum_{i=1}^N p_i |\sin(\pi n t x_i)| \cdot \norma{\cos(2\pi nx \oplus \pi nt x_i )}{} + \frac{\e}{2} \\
	& \leq 2 \sum_{i=1}^N p_i |\sin(\pi n t x_i)| + \frac{\e}{2} \leq  \frac{4\pi}{n_0} \sum_{i=1}^N p_i  + \frac{\e}{2} \leq \e. \qed
\end{align*} 
\end{rem*}

Our result gives the bound for the norm of an operator of the form $(I-A_t)^{-1}$. The main difficulty is that this operator is not globally invertible. Of course, boundedness of a resolvent operator $R_\lambda(A) = (A - zI)^{-1}$ has been thoroughly studied (see e.g. \cite{G}, \cite{Ba} which feature Hilbert space setting for Hilbert-Schmidt and Schatten-von Neumann operators). Let us also mention that the first part of the book \cite{CE} is a set of related articles concerning mainly the problem of finding the inverse formula for certain Toeplitz-type operators. The paper \cite{GS} contains the famous Gohberg-Semencul formula for the inverse of a non-Hermitian Toeplitz matrix. In \cite{GH} the authors generalized the results of \cite{CE} to the case of Toeplitz matrices whose entries are taken from some noncommutative algebra with a unit. The operators of the form $I-K$ (acting e.g. on $L_1([0,1])$), where $K$ is a certain operator with a kernel $k(t-s)$, are continuous versions of the operators given by Toeplitz matrices. Paper \cite{GS} deals also with this kind of operators, namely
\[
	(I-K)(f)(x) = f(x) - \call{0}{1}{k(t-s)f(s)}{s},
\]        
where $k \in L_1([-1,1])$. In the case of $I-K$ being invertible, the authors give a formula for the inverse operator $(I-K)^{-1}$ in terms of solutions of certain four integral equations. See also Article 3 in \cite{CE} for generalizations of these formulas.

\section{Proof of Theorem \ref{thm2}}\label{proof}

We begin with two lemmas.

\begin{lem}\label{lem1}
Suppose $Y$ is an $\ell$-decent random variable. Let $Y_1,Y_2,\ldots$ be independent copies of $Y$. Then there exists a positive integer $N=N(Y)$ and numbers $c=c(Y)>0$, $C_0 = C_0(Y) \geq 1$ such that for all $C \geq C_0$ and $n \geq N$ the random variable
\begin{equation}\label{eq:defXC}
	X_n^{(C)} = \left( C \cdot \frac{Y_1+\ldots+Y_n}{\sqrt{n}} \right) \mod 1
\end{equation}
is $c$-good.
\end{lem}
\begin{proof}

We prove the lemma in a few steps considering more and more general assumptions about $Y$.

\emph{Step I. }
Suppose that the characteristic function of $Y$ belongs to $L_p(\R)$ for some $p \geq 1$. In this case, by a certain version of the Local Central Limit Theorem, e.g. Theorem 19.1 in \cite{BR}, p. 189, we know that the density $q_n$ of $(Y_1 + \ldots + Y_n - n\mb{E}Y)/\sqrt{n}$ exists for sufficiently large $n$, and satisfies
\begin{equation}
\label{eq:densities}
\sup_{x \in \mb{R}} \left| q_n(x) - \frac{1}{\sqrt{2\pi}\sigma}e^{-x^2/2\sigma^2} \right| \xrightarrow[n\to\infty]{} 0,
\end{equation}
where $\sigma^2 = \mathrm{Var}(Y)$. Observe that the density $g_n^{(C)}$ of $X_n^{(C)}$ equals
\[ g_n^{(C)}(x) = \sum_{k \in \mb{Z}} \frac{1}{C}q_n\left( \frac{1}{C}(x + k + \sqrt{n}\mb{E}Y) \right), \qquad x \in [0,1]. \]
Using \eqref{eq:densities}, for $\delta = \frac{e^{-2/\sigma^2}}{\sqrt{2\pi}\sigma}$ we can find $N = N(Y)$ such that
\[ q_n(x) >  \frac{1}{\sqrt{2\pi}\sigma}e^{-x^2/2\sigma^2} - \delta/8, \qquad x \in \mb{R}, \ n \geq N.\]
Therefore, to be close to the maximum of the Gaussian density we sum over only those $k$'s for which $x + k + \sqrt{n}\mb{E}Y \in (-2C, 2C)$ for all $x \in [0,1]$. Since there are at least $C$ and at most $4C$ such $k$'s, we get that
\[ g_n^{(C)}(x) > \frac{1}{C}\frac{1}{\sqrt{2\pi}\sigma}e^{-2/\sigma^2} \cdot C - \frac{1}{C}\frac{\delta}{8}\cdot 4C = \frac{1}{2\sqrt{2\pi}\sigma}e^{-2/\sigma^2}.\]
In particular, it implies that $X_n^{(C)}$ is $c$-good with $c = \frac{1}{2\sqrt{2\pi}\sigma}e^{-2/\sigma^2}$. Thus, in this case, it suffices to set $C_0 = 1$.

\emph{Step II. }
Suppose that the law of $Y$ is of the form $q\mu + (1-q)\nu$ for some $q \in (0,1]$ and some Borel probability measures $\mu, \nu$ on $\R$ such that the characteristic function of  $\mu$ belongs to $L_p(\R)$ for some $p \geq 1$. 
Notice that
\begin{align*}
\mu_{Y_1 + \ldots + Y_N} &= \mu_Y^{\star N} = (q\mu + (1-q)\nu)^{\star N} = \sum_{k=0}^N {N \choose k} q^k(1-q)^{N-k}\mu^{\star k}\nu^{\star (N-k)} \\
&\geq \sum_{k = N_0}^N  {N \choose k} q^k(1-q)^{N-k}\mu^{\star k}\nu^{\star (N-k)} = c_{N,N_0}\left( \mu^{\star N_0} \star \rho_{N,N_0} \right),
\end{align*}
where
\[
\rho_{N,N_0} = \frac{1}{c_{N,N_0}} \sum_{k = N_0}^N  {N \choose k} q^k(1-q)^{N-k}\mu^{\star k-N_0}\nu^{\star (N-k)}
\]
is a probability measure, and
\[
c_{N,N_0} = \sum_{k = N_0}^N  {N \choose k} q^k(1-q)^{N-k}
\]
is a normalisation constant. Choosing $N_0 = \lfloor qN - C_1\sqrt{q(1-q)N} \rfloor$ we can guarantee that $c_{N,N_0} \geq 1/2$ eventually, say for $N \geq \tilde N$. Denoting by $\bar Y$, $Z$ a random variable with the law $\mu$, $\rho_{N,N_0}$ respectively and by $\bar Y_i$ i.i.d. copies of $\bar Y$, we get
\begin{align*}
\p{X_N^{(C)} \in A} \geq c_{N,N_0}\p{\left( C\frac{\bar Y_1 + \ldots + \bar Y_{N_0}}{\sqrt{N}} + C\frac{Z_{N,N_0}}{\sqrt{N}} \right) \mod 1 \in A}.
\end{align*}
By Step I, the first bit $C(\bar Y_1 + \ldots + \bar Y_{N_0})/\sqrt{N}$ is $c$-good for some $c > 0$ and $C \geq C_0^{(II)} = \sup_{N \geq \tilde N} \sqrt{N/N_0}$ . Moreover, note that if $U$ is a $c$-good $\mb{T}$-valued r.v., then so is $U \oplus V$ for every $\mb{T}$-valued r.v. $V$ which is independent of $U$. As a result, $X_N^{(C)}$ is $c/2$-good. 

\emph{Step III.}
Now we consider the general case, i.e. $Y$ is $\ell$-decent for some $\ell \geq 1$. For $n \geq \ell$ we can write 
\[
C \cdot \frac{Y_1+\ldots+Y_n}{\sqrt{n}} = C\sqrt{\frac{\podloga{n/\ell}}{n}}\cdot\frac{\tilde Y_1 + \ldots + \tilde Y_{\podloga{n/\ell}}}{\sqrt{\podloga{n/\ell}}} + C\frac{\tilde R}{\sqrt{n}}
\]
with $\tilde Y_j = Y_{(j-1)\ell + 1} + \ldots + Y_{j\ell}$ for $j = 1, \ldots, \podloga{n/\ell}$, and $\tilde R = Y_{\podloga{n/\ell}\ell + 1} + \ldots + Y_n$. Since the absolutely continuous part of the law $\mu$ of $\tilde Y_j$ is nontrivial, then $\mu$ is of the form $q\nu_1 + (1-q)\nu_2$ with $q \in (0,1]$ and the characteristic function of $\nu_1$ belonging to some $L_p$. Indeed, $\mu$ has a bit which is a uniform distribution on some measurable set whose characteristic function is in $L_2$. Therefore, applying Step II for $\tilde Y_j$'s we get that $X_n^{(C)}$ is $c$-good when $C\sqrt{\frac{\podloga{n/\ell}}{n}} \geq C_0^{(II)}$. So we can set $C_0 = C_0^{(II)}\sqrt{2\ell}$.
\end{proof}

\begin{lem}\label{lem2}
Suppose $Z$ is a $\mb{T}$-valued $c$-good random variable and $B_Z$ is the operator defined by $(B_Z f)(x) = \mb{E}f(x\oplus Z)$. Then for every $1 \leq p \leq \infty$ and every $f \in L_p(\mb{T})$ with $\int_{\mb{T}}f=0$ we have $\norma{B_Z f}{} \leq (1-c) \norma{f}{}$, where $\norma{\cdot}{}$ is the $L_p$ norm.
\end{lem}

\begin{proof}
Fix $1 \leq p < \infty$. Let $\mu$ be the law of $Z$. Define the measure $\nu (A) = (\mu(A) - c|A|)/(1-c)$ for measurable $A \subset \T$. Since $\mu$ is $c$-good, $\nu$ is a Borel probability measure on $\T$. Take $f \in L_p(\mb{T})$ with mean zero. Then by Jensen's inequality we have
\begin{align*}
\norma{B_Z f}{}^p & = \call{0}{1}{ \left| \call{0}{1}{f(x \oplus s)}{\mu(s)}  \right|^p  }{x} \\
&= (1-c)^p\call{0}{1}{ \left| \call{0}{1}{f(x \oplus s)}{\nu(s)}  \right|^p  }{x} \\
& \leq (1-c)^p\call{0}{1}{  \call{0}{1}{|f(x \oplus s)|^p}{\nu(s)}    }{x} \\ 
& = (1-c)^p\norma{f}{}^p \call{0}{1}{}{\nu(s)}  =(1-c)^p \norma{f}{}^p.
\end{align*}

Since $c$ does not depend on $p$ we get the same inequality for $p = \infty$ by passing to the limit.
\end{proof}

\noindent Now we are ready to give the proof of Theorem \ref{thm2}. 

\begin{proof}[Proof of Theorem \ref{thm2}]
Fix $1 \leq p \leq \infty$. Let $Y_1,Y_2, \ldots$ be independent copies of $Y$. Observe that
\begin{align*}
	(A_t^n f)(x) & = \mb{E}f\left( x \oplus tY_1 \oplus \ldots \oplus tY_n)   \right) \\ & = \mb{E}f\left( x \oplus \left( t\sqrt{n}\left(\frac{Y_1 + \ldots + Y_n}{\sqrt{n}} \right) \textrm{mod}  \ 1  \right)  \right).
\end{align*}
Take $n(t)= C_0^2\sufit{1/t^2}N$, where $C_0$ and $N$ are the numbers given by Lemma \ref{lem1}. Therefore, with $X_{n(t)}^{(C)}$ defined by \eqref{eq:defXC}, we can write
\[
(A_t^{n(t)} f)(x)  = \mb{E}f\left( x \oplus X_{n(t)}^{(C)}  \right),
\]
where $C = t \sqrt{n(t)} = t C_0\sqrt{\sufit{1/t^2}N} \geq C_0\sqrt{ N} \geq C_0$. Thus $X_{n(t)}^{(C)}$ is $c(Y)$-good with some constant $c(Y) \in (0,1)$. From Lemma \ref{lem2} we have 
\[
	\norma{A_t^{n(t)} f}{} \leq \left(1-c(Y) \right) \norma{f}{}
\]
for all $f$ satisfying $\int_\mb{T} f=0$. 

The operator $A_f$ is a contraction, namely $\norma{A_t f}{} \leq \norma{f}{}$ for all $f \in L_1(\mb{T})$. Using this observation and the triangle inequality we obtain
\begin{align*}
	\norma{f-A_t f}{} & \geq \frac{1}{n} \left( \norma{f-A_t f}{} + \norma{A_t f-A_t^2 f}{} + \ldots + \norma{A_t^{n-1}f - A_t^{n}f }{}  \right) \\
	& \geq \frac1n \norma{f-A_t^n f}{}.
\end{align*}
Taking $n=n(t)$ we arrive at
\[
	\frac{1}{n(t)} \norma{f-A_t^{n(t)} f}{} \geq \frac{1}{t^{-2}+1} \cdot \frac{1}{C_0^2 \cdot N}\left( \norma{f}{} - \norma{A_t^{n(t)} f}{} \right) \geq \frac{c(Y)}{2C_0^2 \cdot N} t^2 \norma{f}{}.
\]  
It suffices to take $c=c(Y)/(2C_0^2 \cdot N)$.
\end{proof}

\begin{rem*}
Consider an $\ell$-decent random variable $Y$. As it was noticed in the proof of Lemma \ref{lem1} (Step III), the law $Y_1 + \ldots + Y_\ell$ has a bit whose characteristic function is in $L_2$. Conversely, if the law of $S_m = Y_1 + \ldots + Y_m$ has the form $q\mu + (1-q)\nu$ with $q \in (0,1]$ and the characteristic function of $\mu$  belonging to $L_p$ for some $p \geq 1$, then the characteristic function of the bit $\mu^{\star \sufit{p/2}}$ of the sum of $\sufit{p/2}$ i.i.d. copies of $S_m$ is in $L_2$. In particular, that bit has a density function in $L_1 \cap L_2$. Thus $Y$ is $(m\sufit{p/2})$-decent. 
\end{rem*}

\begin{rem*}
The idea to study the operators $A_t$ (see \eqref{eq:defA}) stemmed from the following question posed by Gideon Schechtman (personal communication): \emph{given $\e > 0$, is it true that there exists a natural number $k = k(\e)$ such that for any bounded linear operator $T:L_1[0,1] \to L_1[0,1]$ with $\norma{T}{L_1 \to L_1} \leq 1$ which has the property
\[
\forall f \in L_1[0,1] \ \left( \ |\mathrm{supp} f| \leq 1/2 \implies \norma{Tf}{1} \geq \e \norma{f}{1}
\right) \ \]
there exist $\delta > 0$ and functions $g_1, \ldots, g_k \in L_\infty[0,1]$ such that
\[ \norma{Tf}{1} \geq \delta \norma{f}{1} \quad \textrm{for any $f \in L_1[0,1]$ satisfying $\int_0^1 fg_j = 0$, $j = 1, \ldots, k$}?\]}
This, in an equivalent form, was asked by Bill Johnson in relation with a question on Mathoverflow \cite{MO}. Our hope was that an operator $T = I - A_t$, for some $Y$, would provide a negative answer to Schechtman's question. However, Theorem 2 says that if $Y$ is an $\ell$-decent random variable, then $T$ is nicely invertible on the subspace of functions $f\in L_1$ such that $\int f\cdot 1 = 0$.  
\end{rem*}

\section*{Acknowledgements}

This work was initiated while the authors were visiting the Weizmann Institute of Science in Rehovot, Israel. We thank Prof. Gideon Schechtman for supervision and making our stay possible.

We are grateful to Prof. Krzysztof Oleszkiewicz for his many remarks which led to the present general statement of Theorem \ref{thm2}. We thank Prof. Keith Ball for helping us to simplify the proof of Lemma \ref{lem2}. We also appreciate all the valuable comments Prof. Stanis\l aw Kwapie\' n gave us.

\noindent Piotr Nayar$^\star$, \texttt{nayar@mimuw.edu.pl}

\vspace{1em}

\noindent Tomasz Tkocz$^{\star\dagger}$, \texttt{tkocz@mimuw.edu.pl}

\vspace{1em}

\noindent $^\star$Institute of Mathematics, University of Warsaw, \\
\noindent Banacha 2, \\
\noindent 02-097 Warszawa, \\
Poland. \\

\noindent $^\dagger$Mathematics Institute, University of Warwick, \\
\noindent Coventry CV4 7AL, \\
\noindent UK.

\end{document}